\documentclass{amsart}
\usepackage{amssymb}
\usepackage{latexsym}
\usepackage{euscript}

\usepackage{graphicx}
\usepackage{lmodern}
\usepackage{pifont}

\usepackage[normalem]{ulem}

\usepackage{amsmath,amscd,amssymb,amsthm,amsfonts,color}
\usepackage{latexsym}
\usepackage{euscript}

\usepackage[english]{babel}
\usepackage[T1]{fontenc}
\usepackage[utf8]{inputenc}
\usepackage{xspace}
\usepackage{xcolor}

\usepackage[colorlinks=true, linkcolor=blue, citecolor=red, urlcolor=blue, ]{hyperref}
\usepackage[capitalise]{cleveref}

\crefname{theorem}{Theorem}{Theorems}
\crefname{lem}{Lemma}{Lemmas}
\crefname{proposition}{Proposition}{Propositions}
\crefname{definition}{Definition}{Definitions}
\crefname{cor}{Corollary}{Corollaries}
\crefname{remark}{Remark}{Remarks}

\usepackage{lmodern}
\usepackage{amsbsy}
\usepackage{url}
\usepackage{graphicx}
\usepackage{enumerate}

\newcommand{\taw}{\tau_{\scriptstyle T,B}}
\newcommand{\tawt}{\tau_{ T}}

\newcommand{\rest}{\upharpoonright}
\newcommand{\nkt}{  \n_{\mathfrak{K}} ^2 }
\newcommand{\nk}{\n_{\mathfrak{K}}}

\newcommand{\h}{\mathfrak{H}}
\newcommand{\kb}{\mathfrak{K}_B}
\newcommand{\kbt}{\mathfrak{K}_{T}}
\newcommand{\sK}{\mathfrak{K}}

\newcommand{\la}{\langle}
\newcommand{\ra}{\rangle}
\newcommand{\ikb}{\ra_{\mathfrak{K}_B}}
\newcommand{\ik}{)_{\mathfrak{K}}}

\newcommand{\n}{\|}
\newcommand{\ifaf}{\Leftrightarrow}

\newcommand{\ran}{\mathrm{ran}\,}

\newcommand{\ranbar}{\overline{\mathrm{ran}}\,}
\newcommand{\ranbarT}{\mathrm{\overline{ran}}\, T}

\newcommand{\cran}{\overline{\mathrm{ran}}\,}
\newcommand{\dom}{\mathrm{dom}\,}

\newcommand{\D}{\mathrm{dom}\,}
\newcommand{\DT}{\mathrm{dom}\,T }
\newcommand{\DB}{\mathrm{dom}\,B }

\newcommand{\adj}{^*}

\newcommand{\tstar}{T^{*}}

\newcommand{\half}{^{\frac{1}{2}}}

\newcommand{\halfbt}{( B^*T)^{\frac{1}{2}}}

\newcommand{\N}{\mathbb{N}}

\newcommand{\inclu}{\subseteq}
\newcommand{\converge}{\underset{n \to + \infty} \longrightarrow}

\newcommand{\wt}{\widetilde}

\newcommand{\bh}{\mathbf{B}(\mathfrak{H})}

\usepackage{bm}
\newcommand{\lplus}{\bm{B^{+}}(\mathfrak{H})}
\newcommand{\lplusk}{\bm{B^{+}}(\mathfrak{K})}

\newcommand{\mathscr}{\mathbf}

\newcommand{\bc}{\begin{center}}
\newcommand{\ec}{\end{center}}

\newtheorem{Prop}{Proposition}[section]

\newtheorem{Theo}{Theorem}[section]
\newtheorem{I.I}{prem}
\newtheorem{rem}{Remark}[section]
\newtheorem{lem}{Lemma}[section]
\newtheorem{cor}{Corollary}[section]

\numberwithin{equation}{section}
\numberwithin{equation}{section}

\makeatletter
\renewenvironment{proof}[1][\proofname]{%
   \par\pushQED{\qed}\normalfont%
   \topsep6\p@\@plus6\p@\relax
   \trivlist\item[\hskip\labelsep\bfseries#1\@addpunct{.}]%
   \ignorespaces
}{%
   \popQED\endtrivlist\@endpefalse
}
\makeatother

\begin{document}

\title[Extension of the theorems of Douglas and Sebestyén]
    {An extension and refinement of the \\ theorems of Douglas and Sebestyén\\ for unbounded operators}

\author{Yosra~Barkaoui}
\address{Department of Mathematics and Statistics \\
University of Vaasa \\
P.O. Box 700, 65101 Vaasa \\
Finland}
\email{yosra.barkaoui@uwasa.fi}

\author{Seppo~Hassi}

\address{Department of Mathematics and Statistics \\
University of Vaasa \\
P.O. Box 700, 65101 Vaasa \\
Finland}
\email{sha@uwasa.fi}




\begin{abstract}
  For a closed densely defined operator $T $ from a Hilbert space $\h$ to a Hilbert space $\sK,$ necessary and sufficient conditions are established for the factorization of $T$ with a bounded nonnegative operator $X$ on $\sK.$
  This results yields a new extension and a refinement of a well-known theorem of R.G. Douglas, which shows that the operator inequality $A^*A\leq \lambda^2 B^*B, \lambda \geq 0$, is equivalent to the factorization $A=CB$ with $\|C\|\leq \lambda$. The main results give necessary and sufficient conditions for the existence of an intermediate selfadjoint operator $H\geq 0$, such that $A^*A \leq \lambda H \leq  \lambda^2 B^*B$. The key results are proved by first  extending a theorem of Z. Sebesty\'en to the setting of unbounded operators.
\end{abstract}

\maketitle
%

\section{Introduction}
Let  $(\h, (.,. )_{\h})$ be a Hilbert space and denote by $\bh$  the class of bounded everywhere defined operators on $\h.$ For $T,B \in \bh,$   R.G.~Douglas \cite[Theorem~1]{Douglas1966} showed, in 1966,  that the following equivalence holds for some $\lambda \geq 0:$
\begin{equation}\label{intro1}
  TT^* \leq \lambda^2 \, BB^* \quad \ifaf \quad T= BC , \ C \in \bh \quad  \ifaf \quad \ran T \inclu \ran B.
\end{equation}
Later, in 1983, Z.~Sebestyén \cite{sebestyen1983restrictions} established the following characterization for a related problem:
\begin{equation}\label{intro3}
  T^*T \leq \lambda\, T^*B  \textrm{ for some } \, \lambda \geq 0 \quad \ifaf \quad T= XB , \ X \in \lplus,
\end{equation}
where $T,B \in \bh$ and $\lplus$ stands for the class of bounded nonnegative operators on $\h.$
In fact, the inequality in \eqref{intro3} is closely connected to the one in \eqref{intro1},
since the identity $T=XB$ implies the following two inequalities:
\begin{equation}\label{intro2}
  \tstar T \leq \lambda \tstar B \leq   \lambda^2 B\adj B.
\end{equation}
Therefore, the existence of a product presentation $T=XB$, where the factor $X$ is not only bounded,
but also \textsl{nonnegative} involves an intermediate nonnegative selfadjoint operator $\lambda T^*B$
lying between $T^*T$ and $\lambda^2 B^*B$ in the theorem of Douglas.

      The study of such factorizations has since been extended to more general settings. For instance, in 2013,   D.~Popovici and Z.~Sebestyén \cite[Theorem~2.2]{popovici2013factorizations}  generalized  the second equivalence in  \eqref{intro1} to multivalued linear operators (linear relations) and showed that $T \inclu BC $ for some liner relation $ C $ if and only if $ \ran T \inclu \ran B.$
      On the other hand, the first equivalence in \eqref{intro1} was established by S.~Hassi and H.S.V.~de~Snoo \cite{hassi22014} for both  unbounded linear operators and linear relations in 2015.
      As to \eqref{intro3}, its extension has been recently studied by the present authors in the context of closed unbounded densely defined  operators $T,B$ from the Hilbert space $\h$  to another Hilbert space $\sK$ with domains $\DT$ and $\DB.$
      More precisely, it is shown in \cite[Theorem~2.7]{papertwo} that
        \begin{equation}\label{coreeq}
           \tstar T \leq \lambda \, \tstar B \quad \ifaf \quad X\overline{B_0} \inclu T, \ X \in \lplus,
        \end{equation}
        where $B_0:= B \rest \D ( \tstar B)$ and $\tstar B\geq 0$ is selfadjoint. Moreover, when in addition, $B$ is nonnegative and selfadjoint, also the factorization $T=XB$ is characterized in \cite{papertwo} by means of quasi-affinity to a nonnegative selfadjoint operator. Furthermore, such a factorization is shown to imply several local spectral properties for $T$ studied further in \cite{paper3}.
        In the case of bounded operators Sebestyén's result in \eqref{intro3} has also  been studied in recent papers by M.~L. Arias, G.~Corach, and M.~C. Gonzalez \cite{gustavo2013products} and by M.~Contino, M.~A. Dritschel, A.~Maestripieri, and  S.~Marcantognini \cite{product2021}. In the last paper also some local spectral theoretic results for $T=XB\in \bh$ have been established.

        In this paper, we improve \eqref{coreeq} and establish a complete unbounded analog of \eqref{intro3} without the core condition $B=\overline{B_0}$ and the selfadjointness assumption on $T^*B$. Inspired by the work of Sebestyén, our approach hinges on the construction of an auxiliary Hilbert space
        by means of a sesquilinear form
        $\tau_{\scriptstyle T,B}[f,g] := ( Tf, Bg \ik, $ $ f,g \in \D \taw = \DB \inclu \DT$.
        In the first step, the following equivalences are proved in \cref{theo1}:
             \begin{equation}\label{eqtaw}
             X B \inclu T \ifaf \tawt \leq \lambda \, \taw \  \ifaf \n T f \nkt \leq \, {\lambda}  ( Tf, Bf \ik ,
             \end{equation}
       where the form $\tawt$ is defined by $\tawt[f]=(Tf,Tf)_\sK$. An important further result is that the form $\taw$ is closable and, therefore, its closure gives rise to a nonnegative selfadjoint operator $H$, and the inequality in \eqref{eqtaw} can be described explicitly by means of $H$, $X$ and $B$ as follows:
       \begin{equation}\label{eqH}
           \tstar T \leq \, \lambda H, \quad H=  B\adj X\half \overline{X\half B};
       \end{equation}
    see \cref{theo2}. An essential difference here is that in \eqref{coreeq} the operator $T^*B$ is assumed to be selfadjoint,
    while \eqref{eqH} shows that $T^*B$ is in general just a symmetric restriction of $H=H^*\geq 0$. Moreover, here $\dom B$ is
    a core for the form $\overline{\taw}$ generating the operator $H$.
    A further study of the inequality $\tstar T \leq \, \lambda H$, where $H$ is only assumed to be a selfadjoint operator
    (i.e. without the specific formula for $H$ in \eqref{eqH}), is carried out and yields further equivalent conditions
    for \eqref{eqtaw} in \cref{prop3}, for instance:
     \begin{align*}
       XB \inclu T \ifaf \, \tstar T \leq \, \lambda H, \text{ where } H\inclu B\adj T \text{ and  } \DB \inclu \D H\half.
    \end{align*}
This not only completes the extension of Sebestyén's result to the present unbounded framework, but the above mentioned results
motivate the investigation of the   \emph{reversed}  version of the above inequalities to be studied in \cref{section3}.
Namely, also in the present case of unbounded operators, the inequality $T^*T \leq \lambda H$ implies the inequality $H\leq \lambda B^*B$;
see \cref{cor21}.
This second inequality will be characterized in \cref{theo3} with some further results, completing the study of the inequalities
\eqref{intro1}--\eqref{intro2} in the case of unbounded operators.

\section{An extension of Sebesty\'en's theorem for unbounded operators} \label{section2}

The main purpose  of this article is in fact to solve the above problem and present analogue  characterizations for the factorization of $T$ as in  \eqref{intro3}. Our first approach is inspired by Sebestyén theorem \cite{sebestyen1983restrictions},   which we now extend to the unbounded  case.

\begin{Theo}\label{theo1}
   Let $(\h, (\cdot,\cdot)_{\h}),  (\mathfrak{K}, (\cdot,\cdot)_{\mathfrak{K}})$ be two complex Hilbert spaces and let $T,B:H \rightarrow K$
   be closed densely defined operators. Then the following statements are equivalent for some $\lambda \geq 0:$
   \begin{enumerate}
    \item[\rm{(1)}] $\n T f \nkt \leq \, \lambda \,   ( Tf, Bf \ik$  for all $f \in \DB \inclu \DT ;$
    \item[\rm{(2)}] there exists $X \in \lplusk$ such   that $\n X \nk \leq \lambda $ and  $XB \inclu T$.
  \end{enumerate}
In this case, $X$ can be selected such that $\ran X \inclu \ranbar T$.
\end{Theo}

\begin{proof} Assume that (1) holds. Then $(T f, Bg \ik$, $f,g \in \D B$,
defines a nonnegative sesquilinear form in the Hilbert space $\h$.
Observe, that
  \begin{equation}\label{innerpro2}
    ( Tf,Bf \ik=0 \quad \ifaf \quad Tf =0, \quad f \in \D B.
  \end{equation}
By completing the quotient space $[\D B/(\ker T\cap \D B)]$ one obtains a Hilbert space $\kb$
whose inner product is denoted by $\la  \widetilde{f},\widetilde{g} \ikb$, $\widetilde{f},\widetilde{g}\in \kb$,
such that
  \begin{equation}\label{innerpro}
 \langle  \wt{f},\wt{g} \rangle_{\kb} = (T f, Bg)_\sK , \quad f,g \in \D B.
  \end{equation}

Now let $V: \kb \rightarrow \mathfrak{K}$ be defined by
   $$
   V \widetilde{f} = Tf \quad \text{for all } f \in \DB.
   $$
Then $V$ is a well-defined linear operator by \eqref{innerpro2} and \eqref{innerpro},
and it follows from  (1) that it is  bounded by $\sqrt{\lambda }.$
It is claimed that
   $$
   V^*  Bf = \widetilde{f}  \quad \text{for all } f \in \DB.
   $$
   To see this, let $f,g \in \DB$ and  $\widetilde{g} \in \kb.$  Then,
   $$
   \la   \widetilde{g}  , V^*  B f \, \ikb =  ( V \widetilde{g}, B f    )_{\mathfrak{K}}  = ( Tg, B f )_{\mathfrak{K}} = \la \widetilde{g}, \widetilde{f} \, \ikb,
   $$
   and therefore $ V^*   Bf = \widetilde{f},$ as claimed. Consequently,   $X:=     VV^* \in \lplusk$ and one has  $\n X \nk \leq \lambda  ,$ $\ranbar X \inclu \ranbarT,$ by construction,  and
   $$
  X Bf =   V\widetilde{f} =  Tf \quad \text{for all } f \in \DB.
   $$
  This proves that $X B   \inclu T.$  \\
\hspace*{0.3cm} For the converse,  assume (2) and let $f \in \DB \inclu \DT.$ Then,
\[
    \|T f\|^2 =\|XBf\|^2 \leq \, \|X^{\half}\|^2 \|X^{\half} Bf\|^2 \leq \lambda (Tf,Bf\ik,
\]
which completes the proof of (1).
\end{proof}

Notice  that a  combination of \eqref{innerpro2} and item (1) of \cref{theo1} shows   that if $T \neq 0$ then   $\lambda  > 0.$   Hence,  the form
\begin{equation}\label{from}
  \tau_{\scriptstyle T,B}[f,g] := ( Tf, Bg \ik, \quad f,g \in \DB \inclu \DT,
\end{equation}
is nonnegative. The next theorem shows that $\tau_{T,B}$ is closable and identifies its closure.

\begin{Theo}  \label{theo2}
 Let $T,B:\mathfrak{H} \rightarrow \mathfrak{K}$ be  closed densely defined linear operators such that condition (1) or, equivalently, (2) of \cref{theo1} holds. Then, the form in \eqref{from} is closable and its closure is given by
     \begin{equation}\label{frombar}
    \overline{\tau_{\scriptstyle T,B}}[f,g] := ( \overline{X\half B} f, \overline{X\half B}  g \ik,
                                             = (H\half f , H\half g \ik
    \quad f,g \in \D \overline{\taw},
\end{equation}
where $H= B\adj X\half \overline{X\half B}= H\adj \geq 0$ and $X$ is as in \cref{theo1}.  Here $H$ is the unique representing selfadjoint nonnegative operator of the form $\overline{\taw}$ with $\D \overline{\taw}= \D H\half \inclu \DT.$

Furthermore, with $\tawt[f,g]:= (Tf, T g \ik,$ $f,g \in \DT,$ each of the following statements is equivalent to the items (1) and (2) of \cref{theo1}:
\begin{enumerate}[\rm(1)]
\item $\tawt \leq \lambda  \taw$ for some $\lambda  \geq 0;$
 \item $\tawt \leq \lambda  \overline{\taw}$ for some $\lambda  \geq 0;$
   \item $
    \tstar T \leq \lambda \,    H$ for some $ \lambda \geq 0.$
\end{enumerate}
\end{Theo}

\begin{proof}
 Let $X$ be the nonnegative operator in item (2) of \cref{theo1} and let $f,g \in \D \taw.$ Then,  \eqref{from} yields
  \begin{equation}\label{krib}
   \taw[f,g] = ( Tf, Bg \ik =(  XBf, Bg \ik
                               =(  X\half Bf, X\half Bg \ik
  \end{equation}
and therefore to prove the closability of $\taw$ is equivalent to prove the closability of the associated operator $X\half B$ to $\taw,$ by  \cite[VI, Example 1.23]{kato1980}).
  To see this, let   $(f_n)_{n \in \N}\inclu \DB$ such that $f_n \converge 0$  and   $X\half  B f_n \converge g .$    Since $X\half \in \bh,$ it follows that $ X B f_n \converge X\half g.$ On the other hand, the inclusion $XB \inclu T$ together with the fact that $T$ is closed yields that $X\half g =0.$  Since $g \in \ranbar X\half,$ one concludes that $g=0.$ Thus   $X\half B$  is closable. Consequently, $\overline{ X\half B}$ is a densely defined operator such that $ X\half \overline{X\half B} \inclu \overline{ XB} \inclu T$ and, in particular,
   $$
   \D \overline{    X\half B} \inclu \DT.
   $$
   Furthermore,  the operator $ H:=( X\half B)\adj \overline{ X\half B} \geq 0$ is     selfadjoint and  it follows from \eqref{krib} that
  $$
  \overline{\taw}[f,g]=(  \overline{X\half B} f, \overline{X\half B}g \ik = ( H\half f, H\half g \ik \quad \text{for all } f,g \in \D  \overline{\taw}.
  $$
  One  concludes that $\D \overline{\taw} = \D H\half = \D \overline{X\half B} \inclu \DT.$ \\
  \hspace*{0.3cm} To see the stated equivalences, observe first from  \eqref{from}  that for all $f \in \D \taw=  \DB \inclu \DT  $ and for a fixed $\lambda  \geq 0,$ one has
         \begin{equation}\label{equivalence0}
            \n T f \nkt \leq \, \lambda  \,  ( Tf, Bf )_\sK \quad \ifaf \quad \tawt[f] \leq \lambda  \, \taw[f].
         \end{equation}
   On the other hand, it is clear that $\tawt$ is closed, since $T$ is closed; cf. \cite[VI, Example 1.13]{kato1980}.
   Hence, item (v) of  \cite[Lemma 5.2.2]{behrndt2020boundary} gives
   \begin{equation}\label{formsi}
     \tawt \leq \lambda  \, \taw \Rightarrow \tawt \leq \lambda  \, \overline{\taw}.
   \end{equation}
Furthermore,  $ \taw \inclu \overline{\taw}$ so again by     \cite[Lemma 5.2.2]{behrndt2020boundary}  one has
$\overline{\taw}  \leq \taw.$
This together with  \eqref{formsi} implies that
\begin{equation}\label{formsi2}
     \tawt \leq \lambda  \, \taw \ifaf \tawt \leq \lambda  \, \overline{\taw} \ifaf \tstar T \leq  \lambda \, H,
\end{equation}
see \cite[Theorem 5.2.4]{behrndt2020boundary} (or, \cite[VI, Remark 2.29]{kato1980}).
One concludes the equivalences $\rm(1)- (3)$ by a combination of \cref{theo1} with \eqref{equivalence0} and \eqref{formsi2}.
\end{proof}

\begin{cor}\label{cor21}
  Let the operators $T,B: \h \rightarrow \mathfrak{K}$ satisfy the conditions {\rm (1)} and {\rm (2)} in \cref{theo1}.
  Then the following statements hold for $0 \leq \lambda\, (= \n X\n)$:
  \begin{enumerate}
  \item[\rm(1)] $0\leq (Tf,Bf\ik \leq \, \lambda \,  \|Bf\|^2$ for all $f \in \DB$;

  \item[\rm(2)] the operator $\tstar B$ is symmetric and, moreover,
   \begin{equation}\label{gall}
   \tstar B \inclu  H \inclu  B\adj T;
   \end{equation}
    \item [\rm(3)] $\tstar T \leq \lambda H \leq \lambda^2 B\adj B;$
   \item [\rm(4)]  if $\tstar B$ is selfadjoint, then equalities hold in \eqref{gall} and one has
                            \begin{equation}\label{twoul}
                              \tstar T \leq  \lambda \,    \tstar B =  \lambda \,   B\adj X B \leq  \lambda^2  \, B\adj B.
                             \end{equation}
  \end{enumerate}
\end{cor}

\begin{proof}
(1) Since $XB \inclu T$ one has for all $f\in \DB$,
\begin{equation}\label{new1}
 (Tf,Bf\ik = (XBf,Bf\ik = \|X\half Bf \|_\mathfrak{K}^2 \leq \|X\half\|^2\|Bf \|_\mathfrak{K}^2,
\end{equation}
so that the inequality holds for $\lambda = \|X\|$.

(2) Under the conditions of \cref{theo1} one has  $XB \inclu T$ for $X \in \lplusk,$  and since $\overline{XB} = \overline{X\half \overline{X\half B}}$ one concludes that
 \begin{equation}\label{j1proof}
  \tstar B \inclu B\adj X B \inclu B\adj X\half \overline{X\half B} = H \inclu    B\adj  \overline{X\half \overline{X\half B}} = B\adj \overline{X B} \inclu B\adj T.
 \end{equation}
  Since $H$ is selfadjoint, $\tstar B $ is symmetric and the proof of $\eqref{gall}$ is completed.

(3) Observe from \cref{theo2} that $\DB \inclu \D H\half \inclu \DT$ and let $f \in \DB.$ Then,
$$
\n H\half f\nk^2 = \n \overline{X\half B} f \nk^2 = \n X\half B f \nk^2 \leq \n X\half \nk^2 \n Bf\nk^2,
$$
which shows that $ H \leq \lambda B\adj B$ with $\lambda = \n X \n.$
The other inequality was proved in \cref{theo2}.

(4) Assume that  $\tstar B$ is selfadjoint.  Then  $B\adj T \inclu (\tstar B)\adj= \tstar B,$
so from \eqref{j1proof} one concludes that
     \begin{equation}\label{j1}
       \tstar B=     B\adj X B= H=  B\adj T.
     \end{equation}
This proves the equalities in \eqref{gall} and by item (3) completes the proof.
\end{proof}

Inspired by \cref{cor21}, a natural question arises as to whether items (1) and (3) can also be regarded as sufficient conditions.
As a first step, item (3) will be examined in the next lemma and \cref{prop3} in a more general framework, where the nonnegative operator $H =H\adj$ is assumed to be independent from $X.$ The second step deals with item (1), in particular with the question when the   following implication holds for some $\lambda \geq 0:$
\begin{equation}\label{secondstep}
  0\leq (Tf,Bf\ik \leq \, \lambda \,   \|Bf\|^2 \   \ \Rightarrow \ \n Tf \nkt \leq \,  \lambda \,  (Tf,Bf\ik
\end{equation}
 for all $ f \in \DB \inclu \DT.$
    This question induces the study of a reversed version of Sebestyén inequality appearing in the left-hand side of \eqref{secondstep} and will be further studied in \cref{section3}.

 \begin{lem}\label{addition}
 Let $H=H\adj \geq 0$ and  $T,B: \h \rightarrow \mathfrak{K}$ be closed densely defined operators such that
 $\DB \inclu \D H\half$ and  $ H \inclu B\adj T.$ Then the following implication holds:
  \begin{equation}\label{eq1rem}
  \tstar T \leq \lambda H \Rightarrow \n T f \nkt  \leq\, \lambda\, (Tf,Bf \ik \quad \text{ for all } f \in \DB \inclu \D T.
  \end{equation}
\end{lem}

\begin{proof}
  Assume that
  \begin{equation}\label{hey1}
    \tstar T \leq \lambda \, H
  \end{equation}
  and let $f \in \DB \inclu \D H\half.$ Then, $f \in \DT$ and since $\D H$ is a core for $H\half,$ there exists $(f_n)_{n \in \N} \inclu \D H \inclu \D B\adj T $ such that $f_n \converge f$ and $H\half f_n \converge H\half f$. By \eqref{hey1} this implies  that   $\n T (f_n - f)\nk \converge 0$ and hence $T f_n \converge Tf.$ Thus
  \begin{align*}
    (Tf,Bf \ik & =  \underset{n \to + \infty}{lim} ( T f_n, B f \ik =  \underset{n \to + \infty}{lim} ( B\adj T f_n, f \ik
            =\underset{n \to + \infty}{lim} (H f_n,    f \ik\\
           & = \underset{n \to + \infty}{lim} (H\half f_n,  H\half  f \ik
             =  ( H\half f,  H\half f \ik = \n H\half f \nkt.
  \end{align*}
  Combining this with \eqref{hey1} completes the argument.
\end{proof}

The next proposition gives some further operator theoretic criteria which are equivalent to the  conditions in \cref{theo1}. The proof is directly obtained from a combination of \cref{theo1},  \cref{cor21} and \cref{addition}.

\begin{Prop}\label{prop3}
Let  $T,B: \h \rightarrow \mathfrak{K}$ be closed densely defined operators. Then, the following statements are equivalent:
   \begin{enumerate}[\rm(1)]
   \item $\n T f \nkt  \leq \lambda\, (Tf,Bf \ik $ for all $ f \in \DB \inclu \D T;$
   \item $\n T f \nkt \leq \lambda\, (Tf,Bf\ik \leq \, \lambda^2 \,  \|Bf\nkt$ for all $f \in \DB \inclu \DT$;
   \item $T^*T \leq \lambda H \leq  \lambda^2 B\adj B $ for some $\lambda \geq 0$ and some $0 \leq H= H\adj \inclu B\adj T;$
   \item $T^*T \leq \lambda H  $ for some $\lambda \geq 0$ and $0 \leq H= H\adj \inclu B\adj T$ with  $\DB \inclu \D H\half.$
   \end{enumerate}
   In particular, if $B\adj T$ is selfadjoint then $H = B\adj T.$
\end{Prop}

\cref{prop3} is, in fact,  a useful tool to cover   Sebestyén theorem  in the general case of unbounded operators, as described in the next corollary, which is analogous to Proposition 2.10 and Corollary 2.11 in \cite{papertwo}.

\begin{cor}
Let  $T,B: \h \rightarrow \mathfrak{K}$ be closed densely defined operators. Then, the following statements are equivalent:
   \begin{enumerate}[\rm(1)]
   \item $T = XB$ has a solution $X \in \lplusk;$
     \item  $ \n T f \nkt  \leq \lambda (Tf,Bf \ik$   for all $ f \in \DB= \DT;$
     \item $\tstar T \leq \lambda B\adj T  \leq \lambda^2 B\adj B$ for some $\lambda \geq 0$ and $ \DT  \inclu \DB $;
     \item $ \tstar T \leq \lambda B\adj T   $ for some $\lambda \geq 0$ and $\DT \inclu \DB \inclu \D \halfbt.$
   \end{enumerate}
\end{cor}

\section{Characterization of the reversed inequality} \label{section3}

The second step involving \eqref{secondstep} is now considered. Analogously to \cref{theo1},
the following result characterizes a reversed inequality.

\begin{Theo}\label{theo3}
  Let $(\h, (, )_{\h})$ and $(\mathfrak{K}, (, )_{\mathfrak{K}})$ be complex Hilbert spaces and $T,B:\h \rightarrow \mathfrak{K}$ be closed densely defined operators. Then, the following statements are equivalent for some $m > 0:$
   \begin{enumerate}
    \item[\rm(1b)] $ \|Tf\|^2_\sK \geq m(Tf,Bf)_\mathfrak{K} \geq 0$  for all $f \in \dom T \inclu \dom B;$
    \item[\rm(2b)] there exists $Y  \in \lplusk$ such that $YT \inclu P_T B,$
         where $P_T$ stands for the orthogonal projection onto $\ranbar T$.
  \end{enumerate}
\end{Theo}

\begin{proof}
Consider the sesquilinear form $(Tf,Bg)$, $f,g \in \dom T$.
By assumption the quadratic form $(Tf,Bf)$ is nonnegative for all $f\in\dom T$.
Therefore, it satisfies the Cauchy-Schwarz inequality, i.e.,
\begin{equation}\label{new01}
 |(Tg,Bf)_\sK| \leq (Tf,Bf)_\sK^{\half}(Tg,Bg)_\sK^{\half} \quad \textrm{for all } f,g \in \dom T.
\end{equation}
Hence, if $(Tf,Bf)_\sK =0$ for some $f\in \dom T$ then by \eqref{new01}
$(Tg,Bf)_\sK=0$ holds for all $g\in \dom T$, i.e., $Bf \in (\ran T)^\perp =\ker T^*$.
The converse is also true and, therefore, for $f\in \dom T$ one has $(Tf,Bf)_\sK=0$
if and only if $f\in \ker T^*B$.

Next observe that $P_TB(\dom T)=\{0\}$ if and only if $(Tf,Bf)=0$ for all $f\in \dom T$,
i.e., $(Tf,Bg)$ is the $0$-form on $\dom T$; cf. \eqref{new01}.
In this case (1b) holds trivially and $Y=0$ satisfies the inclusion in (2b), and the equivalence
of (1b) and (2b) holds in this case.

$\rm{(1b)} \Rightarrow \rm{(2b)}$
Assume that $P_TB(\dom T)\neq \{0\}$ or, equivalently, that for some $f \in \dom T$ one has $(Tf,Bf)_\sK>0$.
In particular, in this case also $T\neq 0$.

Now introduce the Hilbert space $(\kbt,\langle .,.\rangle_{\kbt})$ by completing
the factor space $[\dom T/(\dom T\cap \ker T^*B)]$ with respect to the inner product
  \begin{equation}\label{innerprohi}
  \langle \widetilde{f},\widetilde{g} \rangle_{\kbt}:= (T f, Bg)_\sK , \quad f,g \in \dom T,
  \end{equation}
where  $\widetilde{f},\widetilde{g}$ represent the corresponding equivalence classes.

Next, let the mapping $Z: \ran T  \rightarrow \kbt$ be defined by
\[
  Z Tf =\widetilde{f} \quad \text{for all } f \in \dom T.
\]
Then \eqref{innerprohi} shows that $Z$ is a well-defined linear operator which is bounded by $1/\sqrt{m}$,
since the assumption in (1b) implies that
\[
 \|Z Tf\|_{\kbt}^2 = \langle \widetilde{f}, \widetilde{f} \rangle_{\kbt} = (Tf, Bf)_\sK \leq \tfrac{1}{m} \| Tf\|_\sK^2
 \quad \textrm{for all } f \in \dom T.
\]
By continuity $Z$ can be extended to a bounded operator from $\cran T$ to $\kbt$
and with a zero continuation to $(\ran T)^\perp$ one gets a bounded operator $\sK\to\kbt$,
which is still denoted by $Z$.
It is claimed that
\[
   Z^*  \widetilde{f} = P_T Bf  \quad \text{for all } f \in \dom T.
\]
To see this, let $h=Tt$, $t\in \dom T$ and $f \in \dom T.$ Then,
\begin{align*}
 (h , Z^*\widetilde{f} \,)_\sK = (Z Tt, \widetilde{f})_{\kbt} = (\widetilde{t} , \widetilde{f})_{\kbt} = (Tt, Bf)_\sK = (h, Bf)_\sK,
\end{align*}
which proves that  $ Z^*\widetilde{f} - Bf \perp \ran T$. By construction, $(\ran T)^\perp \inclu \ker Z$ and hence $\ran Z^*\inclu \cran T$.
Therefore, $Z^*\widetilde{f}=P_TZ^*\widetilde{f}=P_T Bf$ as claimed.
Thus, for $Y:=     Z\adj Z  \in \lplusk$ one has $ \n Y \n \leq \tfrac{1}{m}$ and
\begin{equation}\label{star}
   Y Tf =  Z\adj Z Tf =  Z\adj \widetilde{f}= P_T Bf \quad \text{for all } f \in \dom T,
\end{equation}
which means that $YT \inclu P_T B.$

\hspace*{0.3cm} $\rm{(2b)} \Rightarrow \rm{(1b)}$
By the first part of the proof, the statement holds trivially if $Y=0$.
Now assume that $Y\neq 0$, so that $M:=\|Y\|>0$.
Then by assumption $YT\inclu P_T B$ and hence for all $f \in \DT$ one has
\[
 (Tf,Bf)_\sK = (Tf,P_T Bf)_\sK =(Tf,YTf)_\sK = \|Y^{\half}Tf\|^2_\sK \leq M\,\|Tf\|^2_\sK,
\]
which completes the proof of (1b) with $m=1/M>0$.
\end{proof}

The proof shows that one can take $M=1/m>0$ in Theorem \ref{theo3} when the form $(T\cdot,B\cdot)$ is nontrivial.

\begin{cor}\label{corfinal}
The inequality (1b) in Theorem \ref{theo3} implies also the following inequality:
\begin{equation}\label{eq2c}
  (Tf,Bf)_\mathfrak{K} \geq m \,\|P_TBf\|^2_\sK \quad \textrm{for all } f \in \dom T.
\end{equation}
If $P_T B$ is closable then also the form $(T\cdot,B\cdot)$ on the domain $\dom T$ is closable,
and this holds, in particular, if $\ran B \subseteq \cran T$, in which case $P_T B=B$.
\end{cor}
\begin{proof}
By Theorem~\ref{theo3} $YT\inclu P_T B$, where $\|Y\|\leq 1/m$; cf. \eqref{star}.
Therefore, one obtains for all $f\in \dom T$,
\[
 \|P_T Bf\|^2_\sK = \|YTf\|^2_\sK \leq \frac{1}{m}\, \|Y^{\half} Tf\|^2_\sK = \frac{1}{m}\, (Tf,YTf)_\sK = \frac{1}{m}\, (Tf,Bf)_\sK,
\]
which gives the inequality \eqref{eq2c}. Notice also that if $(Tf,Bf)=0$ for all $f\in \dom T$ then,
equivalently, $P_TB(\dom T)=\{0\}$ (cf. the proof of Theorem~\ref{theo3}),
so that \eqref{eq2c} remains true also in this case.

The second statement can be proved in the same way as the closability was proven in \cref{theo2}.
By assumption in \cref{theo3} $B$ is closed and hence the last statement is clear, since
$\ran B \subseteq \cran T$ holds precisely when $P_T B=B$.
\end{proof}

\begin{rem}
Using \cref{theo1} and  switching the roles of $T$ and $B$ in \cref{prop3} leads to the following equivalent statements
(with $m=1/\lambda$):
\begin{enumerate}[\rm(1)]
\item $ \tstar T \geq m \, H \geq m^2\,B\adj B $ for some  $ 0 \leq H=H\adj\inclu T\adj B;$
   \item  $\|Tf\nkt \geq  m \, (Bf,Tf\ik = m\, (Tf, Bf \ik \geq  \, m^2 \, \n B f \nkt $ for all $f \in \DT \inclu \DB$;
   \item  $Y_2T \inclu B$ for some  $Y_2 \in \lplusk$ such that $\n Y_2 \nk \leq \lambda$ and $\ranbar Y_2 \inclu \ranbar B.$
\end{enumerate}
The inclusion in {\rm (3)} yields $Y_3  T \inclu  P_{\ranbarT} B$ with $Y_3:= P_{\ranbarT} Y_2 P_{\ranbarT} \in \lplusk,$ which is equivalent by  \cref{theo3} and \cref{corfinal} to
\begin{equation}\label{tih3}
  \|Tf\nkt \geq  m \, (Tf,Bf)_\mathfrak{K} \geq m^2 \, \|P_TBf\|^2_\sK \quad \textrm{for all } f \in \dom T.
\end{equation}
   \end{rem}

Another approach to the inequalities characterized in Theorem~\ref{theo1} and Theorem~\ref{theo2} is studied in \cite{papertwo}
under the assumption that the operator $T^*B$ is selfadjoint; cf. \cite[Theorem 2.7]{papertwo}.
The functional analytic approach in the present paper leads to the factorization of the operator $T$ by means of the  $B$
(instead of the core $B_0=B\upharpoonright\dom T^*B$ of $B$) with a nonnegative bounded operator $X$ in Theorem \ref{theo1}
and to an analogous factorization of the operator $B$ in Theorem~\ref{theo3}.
The approach here is based on the nonnegativity of the form $(T\cdot,B\cdot)$, which is defined on a larger domain $\dom T$
(or $\dom B$) than the domain of $T^*B$ in the case when $T^*B$ (or $B^*T$) is assumed to be selfadjoint. The nonnegative
factors are then obtained by constructing new suitable Hilbert spaces from the nonnegative form $(T\cdot,B\cdot)$ in each case.



\bibliographystyle{siam}


\end{document}